\documentclass[11pt]{article}
\usepackage{amssymb}
\usepackage{}
\usepackage{amsfonts}
\usepackage{mathrsfs}
\usepackage{appendix}
\usepackage{amsfonts,amsmath,amsthm,amssymb,cases,mathrsfs}
\usepackage{graphicx}
\usepackage{color}
\usepackage{graphics}

\theoremstyle{plain}
\newtheorem{lemma}{Lemma}[section]
\newtheorem{definition}[lemma]{Definition}
\newtheorem{theorem}[lemma]{Theorem}
\newtheorem{prop}[lemma]{Proposition}
\newtheorem{Remark}[lemma]{Remark}

\newtheorem{assume}[lemma]{Assumption}

\numberwithin{equation}{section} 
\title{\huge Robust Consumption Portfolio Optimization with Stochastic Differential Utility\thanks{This paper is supported by National Key R\&D Program of China (No.2018YFA0703900) and National Natural Science Foundation of China (Nos.11631004, 11701371, 11871163)
.}}
\author{ \Large \rm Jiangyan Pu \footnote{School of Finance,  Shanghai Lixin University of Accounting and Finance, 
Shanghai 201209, China. (Email:
\texttt{20139048@lixin.edu.cn}
)}
~~~and~~~\Large Qi Zhang \footnote{School of Mathematical Sciences, Fudan
University, Shanghai 200433, China. (Corresponding author. Email:
\texttt{qzh@fudan.edu.cn})} \footnote{Laboratory of Mathematics for Nonlinear Science, Fudan University, Shanghai 200433, China.}
}
\date{}
\allowdisplaybreaks
\begin{document}
\maketitle

\textbf{Abstract.} This paper examines a continuous time intertemporal consumption and portfolio choice problem with a stochastic differential utility preference of Epstein-Zin type for a robust investor, who worries about model misspecification and seeks robust decision rules. We provide a verification theorem which formulates the Hamilton-Jacobi-Bellman-Isaacs equation under a non-Lipschitz condition. Then, with the verification theorem, the explicit closed-form optimal robust consumption and portfolio solutions to a Heston model are given. Also we compare our robust solutions with the non-robust ones, and the comparisons shown in a few figures coincide with our common sense.

\textbf{Key words:} stochastic differential utility, robust control, stochastic
differential games, HJB(I) equation, non-Lipschitz condition, Heston model.


\section{Introduction}

Over the past several decades, the optimal consumption and investment problems have been one of hot topics in finance. Looking back in history, we can see that the research development of this topic is based on the developments of the financial theory and tools, like martingale theory, dynamic programming principe, utility, etc. The portfolio selection problem related work can be traced back to  Markowitz \cite{Markowitz} in 1952.
Later Merton \cite{Merton1969}, \cite{Merton1971} first solved the multiperiod expected utility model in continuous time by using dynamic programming in 1969 and 1971. Then, Cox and Huang \cite{Cox}, Karatzas \cite{Karatzas} developed the general martingale approach to consumption and portfolio selection.
Stochastic differential utility (SDU), another mile stone, was  introduced in 1992 by Duffie and Epstein \cite{Duffie1992} as a continuous time limit of recursive utility of Kreps and Porteus  \cite{Krep}, Epstein and Zin \cite{Epstein}. Since the recursive utility can differentiate between the coefficient of relative risk aversion and the elasticity of intertemporal substitution (EIS), it receives more and more attentions.

Looking back on the literature, there are two main methods to solve the optimal consumption-portfolio problem with SDU.
One is called utility gradient approach (or martingale method). Using this approach, Duffie and Skiadas \cite{Duffie1994}, Schroder and Skiadas \cite{Schroder}  generalized the time-separable expected utility in \cite{Cox} and \cite{Karatzas}, whose basic idea is to utilize market completeness to separate the computation of an optimal consumption plan from the corresponding trading strategy.
Then this approach was extended to the trading constrained case by Schroder and Skiadas \cite{Schroder2}. In 2001, El Karoui, Peng and Quenez \cite{el} used the first-order condition (FOC) in a complete market to include nonlinear wealth dynamics under the generalized stochastic differential utility (GSDU), where FOC for the optimal consumption plan is the essence of the gradient approach. The other method to deal with SDU is the dynamic programming approach (or Hamilton-Jacobi-Bellman (HJB) equation approach). This approach turns out to be very powerful as it provides an explicit characterization of optimal strategy in terms of the unique solution of a PDE. With the help of HJB equation, lots of researchers solved the optimal consumption and portfolio problems with SDU, see e.g. Kraft, Seifried and Steffensen \cite{Kraft},  Kraft, Seifeiling and Seifried \cite{Kraft2} to name but a few. In 2017, Xing \cite{Xing} put forward backward stochastic differential equation (BSDE) method to characterize optimal consumption and investment strategies.
However, the bulk of literature, including \cite{Kraft},
\cite{Kraft2}, \cite{Xing}
just mentioned, assumes that the investors have complete confidence in their models.

In 2000, some significant progresses were made in Anderson, Hansen and Sargent \cite{Anderson}, which set up the so-called robust consumption and portfolio choice problems.
In the model uncertainty formulation, the decision maker has a probabilistic benchmark in the sense that he/she ``believes" that the probability distribution on the state space is a given distribution but is not totally confident about this. Hence the priors for the decision maker are the combination of the ``probabilistic benchmark" on one hand and ``anything can happen" on the other hand.
In particular, the decision maker considers such a circumstance, in which he/she may use the relative entropy to measure a distance between a reference model and a true model, and then choose the ``worst" model of underlying assets to find his optimal consumption and investment rules.
Since then, a growing body of literature, like Maenhout \cite{Maenhout04, Maenhout06}, Daniel and Shied \cite{Schied}, Bo and Capponi \cite{bo-ca}, Yang, Liang and Zhou \cite{Yang} began to analyze implications of model uncertainty for asset pricing and portfolio choices.
 Recently, several extensions about robustness were made in different aspects: (i) optimal investment under correlation, equicorrelation,
 variance-covariance or volatility ambiguity, such as 
 Fouque,  Pun and Wong \cite{Fouque}, Han and Wong \cite{Han}, Ismail and Pham \cite{Ismail}, Pun \cite{Pun2};
 (ii) an economy modelled by a multivariate stochastic volatility model, especially the principle component stochastic volatility (PCSV) model, which nests Heston's model as a special case (in one dimension, i.e., one risky-asset case). PCSV was initiated in Escobar, Gotz, Seco and  Zagst \cite{Escobar} and investigated in Bergen, Escobar, Rubtsov and Zagst \cite{Bergen} and Yan, Han, Pun and Wong \cite{Yan}.
   However, aforementioned works about consumption-investment problems were not investigated under SDU preference. 

Other researchers considered intertemporal consumption and portfolio choice for an investor with SDU preferences and took model uncertainty into account, such as Liu \cite{Liu}, in which the author solved HJB equation by getting numerical solutions, and Ait-Sahalia and Matthys \cite{Ait}, in which the coefficients of risky asset were  determined.
In this paper,
we analyze  that an investor has a preference of Epstein-Zin type and the referred stock model follows an Heston model whose volatility and extra return are driven by an external stochastic factor process. During the period, the investor worries that the model depicting investment opportunities is misspecified in the sense of Anderson, Hansen and Sargent \cite{Anderson}, and thus the investor prefers to seek robust consumption and portfolio.

As we know, the verification theorem is crucial to deduce the HJB equation which gives the possibility to produce an explicit optimal solution to consumption-portfolio problem. To study our model uncertainty control problem with non-Lipschitz SDU preferences, our first step is to extend the verification theorem to the game problem with monotonic condition and two controls.
To realize above ideas, we find that the technique of the stochastic Gronwall-Bellman inequality is an efficient tool. In the original SDU paper \cite{Duffie1992}, this inequality had been used to prove the verification theorem under the assumption that the aggregator is Lipschitz with respect to the variable.  Then Kraft, Seifried and Steffensen \cite{Kraft} improved the method of stochastic Gronwall-Bellman inequality and applied it to the verification theorem with non-Lipschitz but monotonic aggregator.
However, the verification theorem in \cite{Kraft} serving for stochastic control problem with one control cannot be immediately applied to the stochastic game problem with two controls. So the establishment of the verification theorem to the game problem with monotonic condition is one contribution of our paper.



Needless to say, the explicit solution of control problem is one of the main concerned issues in control problems. Another contribution of our paper is that we eventually get the explicit closed-form solutions of optimal robust portfolio and consumption-wealth ratio in an Heston model, by an application of Hamilton-Jacobi-Bellman-Isaacs (HJBI) equation. Moreover, in comparison between our explicit solutions of robust investors and those of non-robust investors,  it turns out that more robust investors are more cautious than the non-robust investors and they put lower proportion of the wealth into the consumption and stock, which coincides with our common sense.


This paper is organized as follows. In Section 2, we prove our verification theorem. In Section 3, the financial market is set up and the optimization problems for both investors of non-robustness and robustness are stated respectively, and the robust consumption and portfolio decisions are given. Finally, we show the applications of our study to Heston model and the comparison results between our robust solutions and non-robust ones in Section 4.

\section{Verification theorem}
In this section, we prove the verification theorem about two controls for the stochastic game problem with non-Lipschitz aggregator, different from \cite{Kraft}, in which the verification theorem for stochastic control problem with one control was analyzed.

To begin with, we state our general settings. Given a complete probability space $(\Omega, \mathscr{F}, \mathbb{P})$ and a fixed terminal time $T < \infty$, let $\{B_s\}_{s\in[0,T]}$ be a $d$-dimensional Brownian motion on the probability space. Denote by $(\mathscr{F}_s)_{0\leq s\leq T}$ the nature filtration generated by $\{B_s\}_{s\in[0,T]}$ with $\mathscr{F}_0$ containing all $\mathbb{P}$-null sets of $\mathscr{F}$.

We consider that the state process ${X}=\{{X}_s\}_{s\in[t,T]}$ takes values in $\Xi \subseteq \mathbb{R}^n$ with the dynamics
\begin{equation}\label{eqofx}
dX_s=b(s,X_s,u_s, v_s)ds+\sigma(s,{X}_s,u_s,v_s)dB_s, \quad {X}_t=x,
\end{equation}
where $b:[0,T] \times \Xi \times U \times V \rightarrow \mathbb{R}^n$, $\sigma:[0,T]\times \Xi \times U \times V \rightarrow \mathbb{R}^{n\times d}$ are measurable functions, $U \in \mathbb{R}^{m}$ and $V \in \mathbb{R}^{l}$, $t \in [0,T]$ is an  initial time, $x \in \Xi$ is an initial state, and $u$, $v$ are controls belonging to the admissible control sets defined below.

 \begin{definition}
 For $t \in [0,T]$,
 let $\mathcal{U}_t=\mathcal{U}_t(x)$, $x \in\Xi$, be the $t$-admissible control set of $(\mathscr{F}_t)_{0\leq t\leq T}$-adapted feedback strategies $u=\{u_s\}_{s\in[t,T]}=\{u(s,X_s)\}_{s\in[t,T]}$ taking values in $U$, and $\mathcal{V}_t$ be the set of all $(\mathscr{F}_t)_{0\leq t\leq T}$-adapted processes $v=\{v_s\}_{s\in[t,T]}$ taking values in $V$.

  \end{definition}


%
%
%
%
%

A function $u\in\mathcal{U}_t$ (resp. $v\in\mathcal{V}_t$) is a control of a maximizing player (resp. minimizing player), and
a game payoff functional for $u$ and $v$ is defined by
\begin{equation}\label{eqofJ}
J(s,x;u,v)=J_s=\mathbb{E}_s\big[\int_s^T f(r,X_r,u_r,v_r, J_r)dr+\Phi(X_T)\big],\quad s \in [t,T],
\end{equation}
where $f:  [0,T] \times \Xi \times U \times V  \times \mathbb{R} \rightarrow \mathbb{R}$ is the measurable intertemporal aggregator, $\Phi: \Xi \rightarrow \mathbb{R}$ is the terminal cost, and
$\mathbb{E}_s $ is the conditional expectation with respect to $\mathscr{F}_s$.

The control problem we concern with is as below:
\begin{equation}\label{pz8}
{\bf{(P)}} \quad\quad \sup_{u\in \mathcal{U}_t}\inf_{v \in \mathcal{V}_t}\mathbb{E}\big[\int_t^T f(s,X_s,u_s, v_s,J_s)ds+\Phi(X_T)\big].
\end{equation}
Then the value function is formally defined by
\begin{equation*}\label{eqofK}
K(t,x)=\sup_{u\in \mathcal{U}_t}\inf_{v \in \mathcal{V}_t} J(t,x;u,v).
\end{equation*}

Here we allow the case that

(A1)~~for any $u \in \mathcal{U}_t$ and $v\in \mathcal{V}_t$, \eqref{eqofx}--\eqref{eqofJ} has a unique strong solution.

For $w\in C^{1,2}([0,T) \times \Xi)$, define
\begin{equation*}
 \mathcal{L}^{u,v}[w](s,x)=  w_s(s,x)+\langle b(s,x,u,v), w_x(s,x)\rangle  +\frac{1}{2}\text{tr}(\sigma(s,x, u,v)^\top w_{xx}(s,x) \sigma(s,x, u,v)).
 \end{equation*}
We suppose that

(A2)~~for any $(s,x) \in [t,T] \times \Xi$ and $u \in U$, $v \in V $,  there exist measurable functions $v^*: [t,T ] \times \Xi \times U \rightarrow V$ and $u^*: [t,T ] \times \Xi\rightarrow U$ such that
\begin{equation*}\label{optimalpoint}
\begin{aligned}
& v^* \in \arg \min_{v\in V} \{ \mathcal{L}^{u,v}[w](s,x)+f(s,x,u,v,w(s,x)) \},\\
&u^*  \in \arg \max_{u \in U} \{ \mathcal{L}^{u,v^*}[w](s,x)+f(s,x,u,v^*,w(s,x)) \}.\nonumber
\end{aligned}
\end{equation*}

From (A2), it is clear that 
$$\mathcal{L}^{u^*,v^*}[w](s,x)+f(s,x,u^*,v^*,w(s,x))=\sup\limits_{u \in U} \inf\limits_{v \in V} \{\mathcal{L}^{u,v}[w](s,x)+f(s,x,u,v,w(s,x))\}.$$
Note that (A1) and (A2) guarantee that \eqref{eqofx}--\eqref{eqofJ} has a unique strong solution in the cases (i) $u_s=u(s,X_s^{t,x})\in\mathcal{U}_t$, $v_s^*=v^*(s,X_s^{t,x},u(s,X_s^{t,x}))\in\mathcal{V}_t$ and (ii) $u_s^*=u^*(s,X_s^{t,x})\in\mathcal{U}_t$, $v_s=v(s,X_s^{t,x},u^*(s,X_s^{t,x}))\in\mathcal{V}_t$.

Similar to \cite{Kraft}, we also assume that the monotonicity holds, i.e.

(A3)~~there exists $C>0$ such that for any $s,x,u,v\in [t,T] \times \Xi\times U\times V $,
\begin{equation*}
f(s,x,u,v,j_1)-f(s,x,u,v,j_2)\leq C(j_1 - j_2), ~ \text{if} ~ j_1 \geq j_2.~~~~~~~~~~~~~~~~~~~~~~~~~~~~
\end{equation*}


 Our verification theorem is based on the following generalized Skiadas' lemma. See Theorem A.2 in \cite{Kraft}.

\begin{theorem}\label{Generalized Skiadas' lemma}
Let $Y =\{Y_t\}_{t \in [0,T]}$ be a right-continuous adapted process with $Y_T =0$ and $\mathbb{E}[\int_0^T|Y_s|ds] < \infty$. Assume that there exist a progressive process $H=\{H_t\}_{t\in[0,T]}$ and a constant $k \in (0, \infty)$ such that
\begin{equation*}
Y_t=\mathbb{E}_t[\int_t^T H_s ds]\  \text {a.s. and}\ H_t  \geq kY_t  \text{~on~} \{ Y_t \leq 0\}  \text{~a.s. for~} t \in [0,T].
\end{equation*}
Then $Y_t \geq 0 $ for all $t \in [0,T]$ a.s.
\end{theorem}

To use the generalized Skiadas' lemma, we further assume

(A4)~~if $w \in C^{1,2}([0,T)\times \Xi)\cap C^{0,0}([0,T]\times\Xi)$ is a solution of the following dynamic programming equation
\begin{numcases}{}\label{dp}
w_t(t,x)+ \sup_{u \in U} \inf_{v \in V} \{ \langle b(t,x;u,v),w_x(t,x)\rangle+\frac{1}{2}\text{tr}(\sigma^\top(t,x,u,v)w_{xx}(t,x)\sigma(t,x,u,v))\nonumber\\
\ \ \ \ \ \ \ \ \ \ \ \ \ \ \ \ \ \ \ \ \ \ \ \ +f(t,x,u,v,w(t,x)) \}=0~~ \text{in}\ [0,T) \times  \Xi,\\
w(T,x)=\Phi(x),~~ x \in  \Xi,\nonumber
\end{numcases}
the local martingale
\begin{equation*}\label{pz2}
\int^\cdot_{t} w_x(s,X_s)^\top \sigma(s,X_s,u_s,v_s)dB_s
\end{equation*}
is a true martingale for any $u \in \mathcal{U}_t$ and $v \in \mathcal{V}_t$;

(A5)~~for any $u \in \mathcal{U}_t$ and $v\in \mathcal{V}_t$, if $(X_s, J_s)_{t \leq s \leq T}$ is the solution of  \eqref{eqofx}--\eqref{eqofJ}, we have $\mathbb{E}[\int^T_{t}|w(s,X_s)-J_s|ds] < \infty$.\\

Then we have the following verification theorem:
\begin{theorem}\label{pz7}
We assume (A1)--(A5).
If the solution $w$ of HJBI equation (\ref{dp}) exists, then $u^*$ and $v^*$ are optimal
in  $\mathcal{U}_t$ and $\mathcal{V}_t$,   $w$ coincides with the value function of the maxmin problem (\ref{pz8}), and minimax identity holds, i.e.
\begin{equation*}
w(t,x) = J(t,x;u^*, v^*)=\sup_{u \in \mathcal{U}_t}\inf_{v \in \mathcal{V}_t}J(t,x;u,v)=\inf_{v \in \mathcal{V}_t}\sup_{u \in \mathcal{U}_t}J(t,x;u,v).
\end{equation*}
\end{theorem}

\begin{proof}
Step 1: For any $u \in \mathcal{U}_t$, we prove
$
J(t,x;u,v^*) \leq w(t,x).
$

%
For any $u\in \mathcal{U}_t$ and $v^* \in \mathcal{V}_t$ , let $\{(\hat{X}_s, \hat{J}_s)\}_{s\in[t,T]}$ be a solution of  \eqref{eqofx}--\eqref{eqofJ} with $u$ and $v^*$. 
By It\^{o} formula, we have
 \begin{eqnarray*}\label{eqito}
w(s,\hat{X}_s) -\hat{J}_s&=&-\mathbb{E}_s[\int_s^T w_t(r,\hat{X}_r)+\langle b(r,\hat{X}_r,u_r,v^*_r), w_x(r,\hat{X}_r)\rangle\nonumber\\
 &&+ \frac{1}{2}\text{tr}(\sigma^\top(r,\hat{X}_r, u_r,v^*_r) w_{xx}(r,\hat{X}_r) \sigma(r,\hat{X}_r, u_r,v^*_r)) + f(r,\hat{X}_r, u_r, v^*_r, \hat{J}_r)dr]\nonumber\\
 &=& -\mathbb{E}_s[\int_s^T \mathcal{L}^{u_r, v^*_r}[w](r,\hat{X}_r) + f(r,\hat{X}_r, u_r, v^*_r, \hat{J}_r)dr].
 \end{eqnarray*}
Hence, it holds that
 \begin{equation*}
 w(s,\hat{X}_s)-\hat{J}_s = \mathbb{E}_s[\int_s^T {\varphi}_r dr],\ \ \ t \leq s \leq T,
 \end{equation*}
 where
 \begin{eqnarray*}
 {\varphi}_r &=& -\{ \mathcal{L}^{u_r, v^*_r}[w](r,\hat{X}_r)+ f(r,\hat{X}_r, u_r, v^*_r, w(r,\hat{X}_r))\}\\
 &&-\{f(r,\hat{X}_r, u_r, v^*_r, \hat{J}_r)-f(r,\hat{X}_r,u_r, v^*_r,  w(r,\hat{X}_r)) \}.
 \end{eqnarray*}
 Using (A2) and \eqref{dp}, we have
 \begin{equation*}
 \begin{aligned}
 &  \mathcal{L}^{u_r, v^*_r}[w](r,\hat{X}_r)+ f(r,\hat{X}_r, u_r, v^*_r, w(r,\hat{X}_r))\\
 \leq & \mathcal{L}^{u^*_r, v^*_r}[w](r,\hat{X}_r)+ f(r,\hat{X}_r, u^*_r, v^*_r, w(r,\hat{X}_r))\\
 =& 0,
 \end{aligned}
 \end{equation*}
 which implies
 \begin{equation*}
{\varphi}_r \geq -\{f(r,\hat{X}_r, u_r,v^*_r, \hat{J}_r)-f(r,\hat{X}_r, u_r,v^*_r, w(r,\hat{X}_r)) \}.
 \end{equation*}
 From (A3), if $\hat{J}_r  \geq w(r,\hat{X}_r)$ it turns out that
\begin{equation*}
f(r,\hat{X}_r, u_r, v^*_r, \hat{J}_r)-f(r, \hat{X}_r, u_r, v^*_r,  w(r, \hat{X}_r))\leq C(\hat{J}_r - w(r,\hat{X}_r)).
\end{equation*}
 Thus if $w(r,\hat{X}_r)-\hat{J}_r \leq 0$ we have
 \begin{equation*}
 \begin{aligned}
 {\varphi}_r & \geq -[f(r,\hat{X}_r,u_r, v^*_r, \hat{J}_r)-f(r,\hat{X}_r,u_r, v^*_r, w(r,\hat{X}_r))]\\
 & \geq C ( w(r,\hat{X}_r)-\hat{J}_r).
 \end{aligned}
 \end{equation*}
Moreover, by (A5), $\mathbb{E}[\int_t^T|w(s,\hat{X}_s)-\hat{J}_s|ds] < \infty$.  Therefore, by Theorem \ref{Generalized Skiadas' lemma} it yields that
\begin{equation*}
w(s,\hat{X}_s)-\hat{J}_s \geq 0,\ \ \ t \leq s \leq T.
\end{equation*}
In particular, taking $s=t$ we have
\begin{equation*}
w(t,x) \geq \hat{J}_t = J(t,x; u,v^*).
\end{equation*}

Step 2: For any $v\in \mathcal{V}_t$, we prove
$
w(t,x) \leq J(t,x;u^*, v).
$

For $u^*\in \mathcal{U}_t$ and any $v\in \mathcal{V}_t$, let $\{(\tilde{X_s}, \tilde{J_s})\}_{s\in[t,T]}$ be a solution of \eqref{eqofx} and \eqref{eqofJ} with $u^*$ and $v$. 
In a similar way to the argument in Step 1, we have
\begin{equation*}
\tilde{J}_s - w(s,\tilde{X}_s)=\mathbb{E}_s[ \int_s^T \tilde{\varphi}_r dr],\ \ \ t \leq s \leq T,
\end{equation*}
where
\begin{eqnarray*}
 \tilde{\varphi}_r&=& \{ \mathcal{L}^{u^*_r, v_r}[w](r,\tilde{X}_r)+ f(r,\tilde{X}_r, u^*_r, v_r, w(r,\tilde{X}_r))\}\\
 &&+\{f(r,\tilde{X}_r, u^*_r, v_r, \tilde{J}_r)-f(r,\tilde{X}_r,u^*_r, v_r,  w(r,\tilde{X}_r)) \}.
\end{eqnarray*}
Using (A2) and \eqref{dp} again, we have
\begin{equation*}
\begin{aligned}
& \mathcal{L}^{u^*_r, v_r}[w](r, \tilde{X}_r)+f(r,\tilde{X}_r,u^*_r,v_r,w(r,\tilde{X_r}))\\
\geq & \mathcal{L}^{u^*_r, v^*_r}[w](r, \tilde{X}_r)+f(r,\tilde{X}_r,u^*_r,v^*_r,w(r,\tilde{X_r}))\\
= & 0,
\end{aligned}
\end{equation*}
which implies
\begin{equation*}
\tilde{\varphi}_r \geq f(r,\tilde{X}_r, u^*_r, v_r, \tilde{J}_r)-f(r,\tilde{X}_r,u^*_r, v_r,  w(r,\tilde{X}_r)).
\end{equation*}
From (A3), if $w(r,\tilde{X_r}) \geq \tilde{J}_r$ it turns out that
\begin{equation*}
f(r,\tilde{X}_r,u^*_r,v_r,w(r,\tilde{X_r}))-f(r,X_r,u^*_r,v_r,\tilde{J}_r)
\leq  C (w(r,\tilde{X_r})-\tilde{J}_r).
\end{equation*}
Thus if $\tilde{J}_r-w(r,\tilde{X_r})\leq 0$ we have
\begin{equation*}
 \begin{aligned}
\tilde{\varphi}_r & \geq f(r,\tilde{X}_r, u^*_r, v_r, \tilde{J}_r)-f(r,\tilde{X}_r,u^*_r, v_r,  w(r,\tilde{X}_r)) \\
 & \geq C (\tilde{J}_r - w(r,\tilde{X_r})).
 \end{aligned}
 \end{equation*}
Moreover, by (A5), $\mathbb{E}[\int_t^T|\tilde{J}_s - w(s,\tilde{X_s})|ds] < \infty$. Therefore, by Theorem \ref{Generalized Skiadas' lemma} it yields that
\begin{equation*}
\tilde{J}_s - w(s,\tilde{X_s}) \geq 0, \ \ \ \ \ t \leq s \leq T.
\end{equation*}
In particular, taking $s=t$ we have
\begin{equation*}
w(t,x)\leq  \tilde{J}_t= J(t,x;u^*, v).
\end{equation*}

Step 3: $u^*\in\mathcal{U}_t$ and $v^*\in\mathcal{V}_t$ are optimal
solutions.

From Steps 1 and 2, we get that $(u^*,v^*)$ is a saddle point of the functional $J(t,x;u, v)$. By Step 1, for any $u\in \mathcal{U}_t$,
$J(t,x;u,v^*) \leq w(t,x)$. Hence for any $u\in \mathcal{U}_t$,
\begin{equation*}\label{eqa1}
\inf_{v\in \mathcal{V}_t}J(t,x;u,v) \leq w(t,x).
\end{equation*}
In other words, we have
\begin{equation*}
\sup_{u \in \mathcal{U}_t}\inf_{v\in \mathcal{V}_t}J(t,x;u,v) \leq w(t,x).
\end{equation*}
By Step 2, for any $v\in \mathcal{V}_t$,
$w(t,x) \leq J(t,x;u^*, v)$. Hence
\begin{equation*}\label{eqa2}
w(t,x) \leq \inf_{v\in \mathcal{V}_t}J(t,x;u^*,v)\leq\sup_{u \in \mathcal{U}_t}\inf_{v\in \mathcal{V}_t}J(t,x;u,v).
\end{equation*}
So
\begin{eqnarray}\label{eqa3}
w(t,x) =\inf_{v \in \mathcal{V}_t} J(t,x;u^*,v)=\sup_{u \in \mathcal{U}_t}\inf_{v \in \mathcal{V}_t}J(t,x;u,v_u).
\end{eqnarray}

On the other hand, by Step 1, we have
\begin{equation}\label{eqpz15}
\inf_{v \in \mathcal{V}_t}\sup_{u \in \mathcal{U}_t}J(t,x;u,v)\leq \sup_{u \in \mathcal{U}_t}J(t,x;u,v^*) \leq w(t,x).
\end{equation}
By Step 2, we have
\begin{equation}\label{eqpz16}
w(t,x) \leq \sup_{u \in \mathcal{U}_t}J(t,x;u,v)\leq \inf_{v \in \mathcal{V}_t}\sup_{u \in \mathcal{U}_t}J(t,x;u,v).
\end{equation}
Combing \eqref{eqpz15} and \eqref{eqpz16}, we have
\begin{equation}\label{eqpz17}
w(t,x)
=\inf_{v\in \mathcal{V}_t}\sup_{u\in \mathcal{U}_t}J(t,x;u,v).
\end{equation}
Hence, by Step 1 and Step 2 again, it yields that
$$
J(t,x;u^*,v^*) \leq w(t,x)\leq J(t,x;u^*, v^*).
$$
So
\begin{eqnarray}\label{eqa4}
w(t,x) = J(t,x;u^*, v^*).
\end{eqnarray}
Combining \eqref{eqa3}, \eqref{eqpz17} with \eqref{eqa4}, we finally have
\begin{eqnarray*}
w(t,x)
 =J(t,x;u^*,v^*)=\sup_{u \in \mathcal{U}_t}\inf_{v \in \mathcal{V}_t}J(t,x;u,v)=\inf_{v \in \mathcal{V}_t}\sup_{u \in \mathcal{U}_t}J(t,x;u,v).
\end{eqnarray*}

 Thus, $u^*$ is the optimal response when the second player chooses $v^*$, and vice versa. Moreover, $w(t,x)$ coincides with the value function of  the maxmin problem (\ref{pz8}).
%
%
%
 \end{proof}

The verification theorem is the key result to solve the robust optimal portfolio and consumption problem in next sections. Even if we leave out the robust settings and only consider the non-robust problem, the proof of above verification is different from that in \cite{Xing} where the comparison theorem of BSDE is applied rather than the dynamic programming principle.

\section{The financial market and robust optimization problem}
In this section, we will introduce the financial model and state our concerned robust optimization problem.
\subsection{The financial market model}
We briefly introduce the financial market in this subsection.
Our discussion is based on the setting of a financial market, in which two assets (or securities) can be traded continuously. One is non-risky asset bond, with its price $\{P_0(s)\}_{s\in [t,T]}$ given by
\begin{equation}\label{eq1}
dP_0(s)=rP_0(s)ds,\quad P_0(t)=1,
\end{equation}
where the constant $r>0$.
The other is risky asset stock, with its price $\{P(s)\}_{s\in [t,T]}$ given by
\begin{equation}\label{eq2}
dP(s)=P(s)[(r+\lambda(s,Y_s^1))ds + \sigma(s,Y_s^1)dB_s].
\end{equation}
Here the stock's excess return and volatility $\lambda$, $\sigma: [0,T]\times \mathbb{R} \longrightarrow \mathbb{R}$ are assumed to be measurable functions depending on time and the state process $\{Y_s^1\}_{s \in [t,T]}$ satisfies
\begin{equation}\label{eq3}
dY_s^1=\alpha(s,Y_s^1)ds+\beta(s,Y^1_s)(\rho dB_s+\sqrt{1-\rho^2}d\hat{B}_s),\quad  Y_t^1=y,
\end{equation}
where $\alpha, \beta:[0,T]\times \mathbb{R} \longrightarrow \mathbb{R}$ are measurable functions, the correlation $|\rho|\leq1$ is a constant, $\hat{B}_s$ and $B_s$ are mutually independent Brownian motions.


We assume that the investor will start with an initial endowment $x>0$ at the time $t$ and try to allocate his wealth into the bond and stock according to a certain strategy at each time $s \in [t,T]$. If we denote by $\{\pi_s\}_{s \in [t,T]}$ the proportion of the wealth invested into the stock, by $\{c_s\}_{s \in [t,T]}$ the consumption decision and by $\{X_s^1\}_{s \in [t,T]}$ the wealth of the agent, the amount of money invested in the bond at the time $s$ is $(1-\pi_s)X_s^1$. In view of \eqref{eq1} and \eqref{eq2}, the non-robust investor's wealth $\{X_s^1\}_{s \in [t,T]}$ satisfies the following equation:
\begin{equation}\label{eqofwealth}
dX_s^1=X_s^1[(r+\pi_s\lambda(s,Y_s^1))ds +\pi_s\sigma_s(s,Y_s^1)dB_s]-c_sds, \quad X_t^1=x.
\end{equation}

%

We assume that the non-robust investor's SDU is a continuous time Epstein-Zin utility as illustrated in (\ref{eqofJ}) which also can be depicted by BSDE
\begin{equation}\label{eqofv}
dV_s=-f(c_s,V_s)ds + Z_s dB_s, \quad V_T =\Phi(X_T^1),
\end{equation}
where
\begin{equation}\label{eq6}
f(c,v)=\delta\theta v\Big[ (\frac{c}{((1-\gamma)v)^{\frac{1}{1-\gamma}}})^{1-\frac{1}{\psi}} - 1 \Big],
\end{equation}
$\delta >0 $ is the rate of time preference, $0< \gamma \neq 1$ is the coefficient of relative risk aversion, $0 < \psi \neq 1$ is the elasticity of intertemporal substitution (EIS), $\theta:=\frac{1-\gamma}{1-\phi}$, $\phi:=\frac{1}{\psi}$, and $\Phi(x):=\epsilon \frac{1}{1-\gamma}x^{1-\gamma}$, $\epsilon\geq0$.

As in Section 2, we only consider that $(c,\pi)$ are the state $(x,y)$ feedback controls. In the absence of model uncertainty, the non-robust investor's problem is
\begin{equation}\label{eq7}
{\bf{(P1)}} \quad \quad \max_{(c,\pi) \in \mathcal{U}^1_t(x,y)}\mathbb{E}\big[\int_t^T f(c_s, V_s)ds+\Phi(X_T^1)\big],
\end{equation}
where the $t$-admissible control set $\mathcal{U}^1_t(x,y)$, $x \in \mathbb{R}^+$, $y \in \mathbb{R}$ is a class of $(\mathscr{F}_t)_{0\leq t\leq T} $-adapted feedback strategies $c_s=\{c(s,X_s^1, Y_s^1)\}_{s\in[t,T]}$ and $\pi_s=\{\pi(s,X_s^1, Y_s^1)\}_{s\in[t,T]}$ taking values in $\mathbb{R}^+$ and $\mathbb{R}$, respectively.


\subsection{Robust optimization problem}
In this subsection, we consider a robust investor, who is not confident with the given reference model, amends his objective, and correspondingly gets the HJBI equation. A robust investor only deems the state dynamics \eqref{eq3} and \eqref{eqofwealth} as a possibly misspecified approximation. He/She wants to consider a family of alternative models which are close to the reference model. Denote by $Z^1=(X^1, Y^1)^\top$ the state vector. The reference model can then be written as
\begin{equation*}\label{eq8}
dZ_s^1=\Theta(s,Z_s^1;c_s,\pi_s)ds+\Lambda(s,Z_s^1;\pi_s)d{\bf{B}}_s,
\end{equation*}
 where $\Theta$ is the drift vector, $\Lambda $ is the volatility matrix of the state vector $Z^1$, $\Sigma=\Lambda\Lambda^\top$
 is the covariance matrix and ${\bf{B}}$ is the Brownian vector ${\bf{B}}_s=\{(B_s, \hat{B}_s)^\top\}_{s\in[0,T]}$. We denote by $\mathbb{F}^{\bf{B}}=(\mathscr{F}^{\bf{B}})_{0\leq s \leq T}$  the filtration generated by ${\bf{B}}$. The vector $\Theta$ and the matrix $\Lambda$ have the appropriate functional forms inherited from \eqref{eq3} and \eqref{eqofwealth}, i.e.
 \begin{equation*}\label{eq8a}
\Theta(s,x,y;c,\pi)=( x(r+\pi\lambda(s,y))-c, \alpha(s,y))^\top
 \end{equation*}
 and
 \begin{equation}\label{eq8b}
\Lambda(s,x,y;\pi)=\left(
  \begin{array}{cc}
   x\pi\sigma(s,y) &  0  \\
   \rho\beta(s,y) & \sqrt{1-\rho^2} \beta(s,y) \\
  \end{array}
\right).
\end{equation}

We assume that the true state evolution equation $Z$ is
 \begin{equation}\label{eq9}
dZ_s=\Theta(s,Z_s;c_s,\pi_s)ds+\Lambda(s,Z_s,\pi_s)[\Lambda(s,Z_s;\pi_s)^\top v_sds + d{\bf{B}}_s],\ \ \ \ \ Z_t=(x,y)^\top.
\end{equation}
  Here  the $t$-admissible control set $\mathcal{U}_t(x,y)$, $x \in \mathbb{R}^+$, $y \in \mathbb{R}$ is a class of
  $\mathbb{F}^{\bf{B}}$-adapted feedback strategies $c_s=\{c(s,X_s, Y_s)\}_{s\in[t,T]}$ and $\pi_s=\{\pi(s,X_s, Y_s)\}_{s\in[t,T]}$ taking values in $\mathbb{R}^+$ and $\mathbb{R}$, respectively. $\mathcal{V}_t$ is the space of all $\mathbb{F}^{\bf{B}}$-adapted processes $v=\{v_s\}_{s\in[t,T]}=\{(v_1(s), v_2(s))^\top\}_{s\in[t,T]}$ taking values in $\mathbb{R}^2$.
 The state equation is much different from that in \cite{Kraft}, i.e. there is an added nonlinear drift term $\Lambda(s,Z_s;\pi_s)\Lambda(s,Z_s;\pi_s)^\top v_s$. Usually we call $v$ the ``distortion" of the true model relative to the approximation model.

\begin{Remark}
Assume that $\mathbb{P}$ is the subjective probability measure under the reference model. Set
\begin{equation*}\label{9a}
(\frac{d\mathbb{Q}}{d\mathbb{P}})_s = \Gamma_s,
\end{equation*}
where
\begin{equation*}\label{9b}
\frac{d\Gamma_s}{\Gamma_s}=\Lambda(s,Z_s;\pi_s)^\top v_sd{\bf{B}}_s, \quad \Gamma_t=1.
\end{equation*}
If the above exponential martingale satisfies Novikov condition, it follows from Girsanov's theorem that $\mathbb{Q}$ is the probability measure under the alternative model. The interpretation is that the investor endogenously chooses an alternative belief about the dynamics of the state variables and accordingly gets the optimal consumption and portfolio policies.
\end{Remark}

The decision maker wants a decision rule which will work well across a set of distortion $v$ close to $0$. We adjust the cost functional in \eqref{eq7}
by appending a penalty term $\frac{1}{2\eta}v^\top\Sigma v$, among which $\eta$ 
is a positive parameter. In this manner, the objective becomes
\begin{equation*}\label{eq10}
{\bf{(P2)}} \quad\quad \max_{(c,\pi) \in \mathcal{U}_t(x,y)}\min_{v \in \mathcal{V}_t}I(t,(x,y);(c,\pi),v),
\end{equation*}
where
\begin{equation}\label{pz3}
I(s,(x,y);(c,\pi),v)=I_s=\mathbb{E}_s\big[\int_s^T (f(c_r,I_r)+{\frac{1}{2\eta_r}}v^\top_r \Sigma_r v_r)dr+\Phi(X_T)\big],\quad s \in [t,T].
\end{equation}

\begin{Remark}
The robust control criterion comes from that the planner does not totally trust his reference model.
The positive parameter $\eta$ indexes the amount of robustness. By setting $\eta = 0$,
the infimum is achieved by the original measure $\mathbb{P}$.
A larger value of $\eta$ strengthens the incentives to be robust.
We refer to \cite{Anderson} for further related discussion.
\end{Remark}

\begin{Remark}
The term  $\frac{1}{2\eta}v^\top \Sigma v$ in objective is actually the discounted relative entropy which quantifies the penalty term. We refer to the existing robust models such as \cite{Liu}, \cite{Maenhout06} and \cite{Skiadas} for such a penalty term, among which \cite{Liu} also focused on the Epstein-Zin utility. It is an interesting but unsolved problem if the penalty term in the recursive utility could be also recursive, and we aim to study for it in the future.
\end{Remark}

We further give specific assumptions to realize (A3).
\begin{assume}\label{gamma}
Let the coefficients $\gamma$ and $\psi$ satisfy one of the following four cases:
\begin{eqnarray*}
 &(a)& \gamma>1  \text { and } \psi>1,\\
 &(b)& \gamma>1  \text { and } \psi<1,  \text { with } \gamma\psi \leq 1,\\
 &(c)& \gamma<1  \text { and } \psi<1,\\
 &(d)& \gamma<1  \text { and } \psi>1,  \text { with } \gamma\psi \geq 1.
\end{eqnarray*}
\end{assume}
Proposition 3.2 in \cite{Kraft} has shown that (A3) is satisfied for the Epstein-Zin utility when Assumption \ref{gamma} holds.
Therefore, set
\begin{equation*}\label{eqofhatL}
\hat{\mathcal{L}}^{(c,\pi)}[w] = w_t +x(r+\pi \lambda)w_x - cw_x +\frac{1}{2}x^2\pi^2\sigma^2 w_{xx}+\alpha w_y + \frac{1}{2}\beta^2w_{yy}+x \pi \sigma\beta\rho w_{xy},
\end{equation*}
and then we can immediately obtain HJBI equation for problem {\bf{(P2)}} according to Theorem \ref{pz7}.
\begin{prop}\label{pz5}
Let $w \in C^{1,2}([0,T]\times (0,\infty)\times \mathbb{R})$ be a solution of the HJBI equation
\begin{numcases}{}\label{eq11}
\max_{(c,\pi)\in (\mathbb{R}^+,\mathbb{R})}\inf_{v \in \mathbb{R}^2}\Big\{f(c,w(t,x,y))+\hat{\mathcal{L}}^{(c,\pi)}[w](t,x,y)+v^\top_t\Sigma(t,x,y;\pi) \big(w_x(t,x,y),w_y(t,x,y)\big)^\top \nonumber\\
\ \ \ \ \ \ \ \ \ \ \ \ \ \ \ \ \ \ \ \ \ +\frac{1}{2\eta}v^\top_t\Sigma(t,x,y;\pi) v_t\Big\}=0,\nonumber\\
w(T,x,y)=\epsilon\frac{1}{1-\gamma}x^{1-\gamma},
\end{numcases}
where the aggregator $f$ is given by \eqref{eq6} and $\Sigma=\Lambda \Lambda^\top$ is given by \eqref{eq8b}.
Assume\\
(i) for any $(c,\pi)\in{\mathcal{U}}_t(x,y)$ and $v \in \mathcal{V}_t$, (\ref{eq9})--(\ref{pz3}) has a unique strong solution $(Z_s, I_s)_{t \leq s \leq T}$;\\
(ii) there exist admissible $v^* \in \mathcal{V}_t$ and $(c^*,\pi^*) \in \mathcal{U}_t(x,y)$ satisfying
\begin{equation*}
\begin{split}
 &\hat{\mathcal{L}}^{(c^*,\pi^*)}[w](s,x,y)+f(c^*, w(s,x,y))+ (v^*_s)^\top\Sigma(s,x,y;\pi^*) \big(w_x(s,x,y),w_y(s,x,y)\big)^\top\\
 &+ \frac{1}{2\eta}(v^*_s)^\top\Sigma(s,x,y;\pi^*) v^*_s=0, ~~\text{for all $(s,x,y)\in [t,T]\times \mathbb{R}^+ \times \mathbb{R}$;}
\end{split}
\end{equation*}
(iii) Assumption \ref{gamma} holds;\\
(iv) the local martingale
\begin{equation*}
\int_t^\cdot w_z(s, Z_s)^\top\Lambda(s, Z_s;\pi_s)d{\bf{B}}_s
\end{equation*}
is a true martingale for any $(c, \pi) \in \mathcal{U}_t(x,y)$ and $v \in \mathcal{V}_t$;\\
(v) for any $(c,\pi)\in{\mathcal{U}}_t(x,y)$ and $v \in \mathcal{V}_t$,  we have $\mathbb{E}[\int^T_{t}|w(s,Z_s)-I_s|ds] < \infty$. 
\\
Then
$(c^*,\pi^*,v^*)$ is an optimal control and $w$ is the value function of problem {\bf{(P2)}}.
\end{prop}

\subsection{Robust consumption and portfolio decisions}
In this subsection, we will get the optimal robust consumption and portfolio by FOC method, and also give the candidate $w(t,x,y)$ for the solutions of HJBI equation \eqref{eq11}.
By \eqref{eq11},  the FOC with respect to $v$ is
 \begin{equation}\label{foc}
 v^*=-\eta \partial w.
 \end{equation}
 Substituting the FOC \eqref{foc} back into the HJBI equation \eqref{eq11}, we have
 \begin{equation}\label{eq12}
 0= \sup_{(c,\pi)\in (\mathbb{R}^+,\mathbb{R})}\{f(c,w)+\hat{\mathcal{L}}^{(c,\pi)}[w] - \frac{\eta}{2}(x^2\pi^2\sigma^2w^2_x + 2x\pi\sigma\beta\rho w_xw_y + \beta^2w^2_y) \}
 \end{equation}
 with boundary condition $w(T,x,y)=\epsilon \frac{1}{1-\gamma}x^{1-\gamma}$, where the aggregator $f$ is given by \eqref{eq6}. With the aid of a constant $k$, we conjecture
\begin{equation}\label{eq13}
 w(t,x,y)=\frac{1}{1-\gamma}x^{1-\gamma}g(t,y)^k, \quad (t,x,y)\in [0,T]\times \mathbb{R}^+ \times\mathbb{R}
 \end{equation}
 with $g \in C^{1,2}([0,T]\times \mathbb{R})$ and $g(T,y)=\epsilon^{\frac{1}{k}}$.
 To explicitly solve the model, we follow Maenhout's idea (see \cite{Maenhout04}, \cite{Maenhout06}). In particular, we assume that the preference parameter $\eta$ is state-dependent and scaled by the value function:
 \begin{equation}\label{eq14}
 \eta(t,x,y)=\frac{a}{(1-\gamma)w(t,x,y)},
 \end{equation}
where $a>0$ is a constant.

 \begin{Remark}
 The parameter $a$ can be interpreted as the preference for robustness.
 \end{Remark}

 In order to make nonlinear PDE \eqref{eq12} become a linear PDE under the conjectures \eqref{eq13} and \eqref{eq14}, we need to make the following assumptions on the coefficients $k$ and $\psi$.

\begin{assume}\label{assume2}
The coefficients $k$ and $\psi$ satisfy the following:
\begin{equation*}\label{eq17}
 k=1\Big/(\frac{(1-\gamma-a)^2 \rho^2}{(\gamma+a)(1-\gamma)}+1-\frac{a}{1-\gamma})
\end{equation*}
 and
\begin{equation*}\label{eq19a}
 \psi=2-\gamma-a+\frac{(1-\gamma-a)^2}{\gamma+a}\rho^2.
\end{equation*}
\end{assume}
\begin{Remark}
The forms of coefficients $k$ and $\psi$ come from the calculus in Proposition \ref{4.2}. Due to the different state equation (\ref{eq9}) the calculus is much different from the case in Seifried and Steffensen \cite{Kraft}. As $a=0$ the forms of coefficients $k$ and $\psi$ coincide with \cite{Kraft}.
\end{Remark}

The following result provides a solution to Problem {\bf{(P2)}}.
\begin{prop}\label{4.2}
 Under Assumptions \ref{gamma}, \ref{assume2} and (i)(ii) in Proposition \ref{pz5}, the candidate for the solution of HJB equation \eqref{eq12} is
 \begin{equation}\label{pz4}
 w(t,x,y)=\frac{1}{1-\gamma}x^{1-\gamma}g(t,y)^k, \quad (t,x,y)\in [0,T]\times \mathbb{R}^+\times\mathbb{R}.
 \end{equation}
Here $g$ solves the following PDE:
\begin{equation}\label{eq18a}
 g_t(t,y)+H_1(t,y)g(t,y)+H_2(t,y)g_y(t,y)+\frac{1}{2}\beta^2(t,y)g_{yy}(t,y)+ \delta^\psi =0,
 \end{equation}
 among which
 \begin{eqnarray}
  H_1 &=&\frac{1}{k}\Big[ (1-\gamma)r+\frac{(1-\gamma)\lambda^2}{2(\gamma+a)\sigma^2}-\delta \theta \Big], \label{eq19aa}\\
  H_2 &=&\frac{(1-\gamma-a)}{(\gamma+a)}\frac{\beta\rho\lambda}{\sigma}+\alpha. \label{eq19bb}
 \end{eqnarray}

Moreover, for this $w$ in (\ref{pz4}), if we assume that 
the local martingale
$ \int_t^\cdot w_z(s, Z_s)^\top\Lambda(s, Z_s;\pi_s)d{\bf{B}}_s $
is a true martingale and $\mathbb{E}[\int^T_{t}|w(s,Z_s)-I_s|ds] < \infty$ for every $(c, \pi) \in \mathcal{U}_t(x,y)$, $v \in \mathcal{V}_t$,
the optimal consumption-wealth ratio and investment allocation are
  \begin{equation} \label{eq16}
 (\frac{c}{x})^*=\frac{\delta^\psi}{g}\ \ {\rm and}\ \  \pi^*=\frac{\lambda}{(\gamma+a)\sigma^2}+\frac{(1-\gamma-a)k\beta\rho}{(\gamma+a)(1-\gamma)\sigma}(\frac{g_y}{g}).
\end{equation}
\end{prop}

\begin{proof}
Since
\begin{equation*}
w(t,x,y)=\frac{1}{1-\gamma} x^{1-\gamma}g^k(t,y),
\end{equation*}
we have
\begin{equation*}
\begin{aligned}
\frac{w_t}{w}& = k \frac{g_t}{g},\\
\frac{w_x}{w}& = (1-\gamma)\frac{1}{x},\\
\frac{w_{xx}}{w}& = -\gamma (1-\gamma)\frac{1}{x^2},\\
\frac{w_y}{w}& =k \frac{g_y}{g},\\
\frac{w_{yy}}{w}&  = k(k-1)\frac{g_y^2}{g^2} + k \frac{g_{yy}}{g},\\
\frac{w_{xy}}{w}&=(1-\gamma)\frac{k}{x}\frac{g_y}{g}.
\end{aligned}
\end{equation*}
Substituting \eqref{eq13} and \eqref{eq14} into \eqref{eq12} and multiplying ${1\over{w}}$ on both sides of the equation, HJB equation \eqref{eq12} becomes
\begin{equation}\label{eq15}
\begin{split}
 0=&\max_{(c,\pi)\in (\mathbb{R}^+,\mathbb{R})} \{k\frac{g_t}{g}+(1-\gamma)(r+\pi \lambda)-(1-\gamma)\frac{c}{x} -\frac{1}{2}(\gamma + a)(1-\gamma)\pi^2\sigma^2 \\
 &+\alpha k\frac{g_y}{g}+\frac{1}{2}\beta^2k(k-1-\frac{ak}{1-\gamma})(\frac{g_y}{g})^2 + \frac{1}{2}\beta^2k\frac{g_{yy}}{g}\\
 &+ (1-\gamma-a)k\pi\sigma\beta\rho\frac{g_y}{g}+ \delta\theta((\frac{c}{x})^{1-\phi}g^\zeta-1)   \},
\end{split}
\end{equation}
where $\zeta:=-\frac{k}{\theta}$. The FOC of HJB equation \eqref{eq15} with respect to $\pi$ and $\frac{c}{x}$ leads immediately to  the optimal investment allocation and consumption-wealth ratio
\begin{equation}\label{optimalpi}
\pi=\frac{\lambda}{(\gamma+a)\sigma^2}+\frac{1-\gamma -a}{(\gamma+a)(1-\gamma)}\frac{\beta \rho k}{\sigma}\frac{g_y}{g}
\end{equation}
and
\begin{equation}\label{optimalc/x}
\frac{c}{x}=(\delta g^\zeta)^\psi.
\end{equation}
Then substituting \eqref{optimalpi} and \eqref{optimalc/x} into HJB equation \eqref{eq15} and checking the coefficient of the term $(\frac{g_y}{g})^2$, we have
\begin{equation*}
\begin{aligned}
& -\frac{1}{2}(\gamma +a) (1-\gamma )\sigma^2 (\frac{1-\gamma -a}{(\gamma+a)(1-\gamma)}\frac{\beta \rho k}{\sigma})^2 + \frac{1}{2}\beta^2k(k-1-\frac{ak}{1-\gamma}) \\
&+ (1-\gamma -a)\sigma \beta \rho k (\frac{1-\gamma -a}{(\gamma+a)(1-\gamma)}\frac{\beta \rho k}{\sigma})\\
=& \frac{1}{2}\beta^2 k (\rho^2 k \frac{(1-\gamma-a)^2}{(\gamma +a )(1-\gamma)} + k-1 -\frac{a}{1-\gamma} k).
\end{aligned}
\end{equation*}
Under our assumption
\begin{equation*}
 k=1\Big/(\frac{(1-\gamma-a)^2 \rho^2}{(\gamma+a)(1-\gamma)}+1-\frac{a}{1-\gamma}),
\end{equation*}
the term $(\frac{g_y}{g})^2$ disappears. Hence 
\begin{equation}\label{eq18}
g_t(t,y)+H_1(t,y)g(t,y)+H_2(t,y)g_y(t,y)+\frac{1}{2}\beta^2(t,y)g_{yy}(t,y)+ \frac{\phi\theta\delta^\psi}{k}g^{\zeta\psi+1}(t,y)=0,
\end{equation}
where $H_1$ and $H_2$ are given in \eqref{eq19aa} and \eqref{eq19bb}. Let $\psi=2-\gamma-a+\frac{(1-\gamma-a)^2}{\gamma+a}\rho^2$, which is equivalent to  $\zeta\psi+1=0$. Then PDE \eqref{eq18} reduces to PDE \eqref{eq18a}. The proof is complete.
\end{proof}
\begin{Remark}
If the coefficients in \eqref{eq18a} satisfy the conditions of Heath and Schweizer \cite{Heath}, the solution $g$ admits the Feynman-Kac representation
 \begin{equation*}\label{eq19b}
 g(t,y)=\delta^\psi H(t,y)+\epsilon^{\frac{1}{k}}h(t,y;T),
 \end{equation*}
 where $H(t,y):=\int_t^T h(t,y;s)ds$, $h(t,y;s):=\tilde{E}_{t,y}[e^{\int_t^sH_1(Y_u)du}]$
 for $t \in [0,T]$, $y \in \mathbb{R}$, and the expectation $\tilde{E}_{t,y}[\cdot]$ is taken with respect to the equivalent measure $\tilde{\mathbb{P}}$, for which $Y$ is conditioned on $Y_t=y$ and has drift $H_2$ instead of $\alpha$ in (\ref{eq3}).
 Here we also note that $h$ satisfies the PDE
 \begin{equation}\label{eq19c}
h_t+H_1h+H_2h_y+\frac{1}{2}\beta^2h_{yy}=0\quad \text{ on } [0,s]\times \mathbb{R}
 \end{equation}
 with the terminal condition $h(s,y;s)=1$.
\end{Remark}

 \section{Application: Heston model}

In this section, we focus on the stochastic volatility model, in which the volatility follows a square-root form suggested by Heston \cite{Heston}. This kind of model is named by Heston and appears in many mathematical finance problems. We name but a few existing studies for the non-robust consumption-portfolio problem of Heston model. For the time-separable utilities, it has been studied by Kraft \cite{kra} without consumption involved in the utility and by Liu \cite{Liu2} with the assumption of zero correlation. For the recursive
utilities, it has been studied by Chacko and
Viceira \cite{cha} with unit EIS and by Xing \cite{Xing} in the case of $\gamma>1$ and $\psi>1$. Here setting $\Phi(x)\equiv0$ in (\ref{eqofv}) we consider the robust consumption-portfolio problem of Heston model for the Epstein-Zin utility without the assumption of zero correlation.



 To be more precise, the dynamics of the risky asset is given by
 \begin{equation*}\label{eq20}
 dP(s)=P(s)[(r+\bar{\lambda}Y_s^1)ds+\sqrt{Y_s^1}dB_s]
 \end{equation*}
 with the constant $r>0$, and the state process follows
\begin{equation*}\label{eq21}
dY_s^1=(\nu-m Y_s^1)ds+\bar{\beta}\sqrt{Y_s^1}(\rho dB_s + \sqrt{1-\rho^2}d\hat{B}_s),
\end{equation*}
where $\nu, m, \bar{\beta} >0$ are all constants and $Y_t^1=y$ is set to be nonnegative. By comparison theorem (see e.g. Theorem 1.5.5.9 in \cite{Monique}), $Y_s^1\geq0$ for all $s\geq t$ a.s. Then in this case equation \eqref{eq9} becomes
 \begin{equation}\label{robustwealthofx}
dX_s=X_s[(r+\pi_s\bar{\lambda}Y_s - C_s)ds + (\pi_s^2 v_1(s)X_s Y_s+ \pi_s \rho\bar{\beta}v_2(s)Y_s )ds+\pi_s\sqrt{Y_s}dB_s],
\end{equation}
\begin{equation*}\label{robustwealthofy}
dY_s=(\nu - m Y_s + \rho\bar{\beta}\pi_sv_1(s) X_s Y_s+\bar{\beta}^2 v_2(s) Y_s)ds +\bar{\beta}\sqrt{Y_s}(\rho dB_s+\sqrt{1-\rho^2}d\hat{B}_s).
\end{equation*}
%

In the Heston model, we not only give the candidate of the optimal solution, but explicitly  check the martingale condition of Proposition \ref{pz5}, i.e. the local martingale $\int_t^\cdot w_z^\top(s, Z_s)\Lambda(s, Z_s;\pi_s)d{\bf B}_s $
is a true martingale and $\mathbb{E}[\int_t^T|w(s,Z_s)-I_s|ds] < \infty$ for every admissible control $(c, \pi; v)$.
For this, recalling $\mathcal{U}_t(x,y)$ and $\mathcal{V}_t$  appearing   below (\ref{eq9}),
we define the following admissible control set for some $ K\in\mathbb{R}^1$ and nonnegative $a_1$ :
\begin{equation*}
\begin{aligned}
 \mathcal{A}_t:=\{ &(c,\pi;v):\ (c,\pi)\in\mathcal{U}_t(x,y),\ v\in\mathcal{V}_t,\ \frac{c}{x}(s,y) \leq \frac{1}{T-s}+by+K;\ \pi\ is\ nonnegative\\
 &and\ bounded\ with\ upper\  bound\  K_{\pi};\ v=(v_1, v_2),\ v_1=-\frac{a_1}{x}, \ and\ v_2 \ is \ nonpositive \\
 & and \ bounded \ \ a.e.\ ;
  X_s\geq 0,\ s\in[t,T]\}. 
\end{aligned}
\end{equation*}
Obviously, when $(c,\pi;v)\in\mathcal{A}_t$, $Y_s\geq0$ for all $s\geq t$ a.s by comparison theorem again. Moreover, we need some added assumptions for the coefficients:\\
(H1) ~$1 < \gamma < \min\{k+2,{1\over q}+1\}$ where $q>2$ such that $q-2$ is sufficiently small, and $\rho \leq 0, ~\bar{\lambda} >0$,\\
(H2) ~$m>\max\{\frac{1-\gamma-a}{\gamma+a}\bar{\lambda}\bar{\beta}\rho, \bar{\beta}{K}_\pi(2(\gamma-1)+a_1)
\}$,\\
(H3) ~$4(\gamma-1)\bar{\beta}^2\tilde{b} < (m-\bar{\beta}{K}_\pi(2(\gamma-1)+a_1))^2$.\\
Here the constants ${K}_\pi$, $\tilde{b}$, $\kappa$ and $b$ are defined by
\begin{equation*}
\begin{aligned}
\tilde{b}&:= b+(\frac{2\gamma-1}{2}+a_1)K_\pi^2, ~ \ K_\pi:=\frac{\bar{\lambda}}{\gamma+a}+\frac{1-\gamma-a}{(\gamma+a)(1-\gamma)}k\bar{\beta}|\rho|\frac{b}{\kappa},\\
{\kappa}&:=m-\frac{1-\gamma-a}{\gamma+a}\rho \bar{\lambda}\bar{\beta}, ~ \ b:=-\frac{1}{2k}\frac{1-\gamma}{\gamma+a}\bar{\lambda}^2.
\end{aligned}
\end{equation*}



\begin{Remark}\label{remark}
In the following Theorem \ref{th1}, we will  show that $(\frac{c}{X})^*, \pi^* \in \mathcal{A}_t$. If  $a_1=a$, then $v_1(t)=v_1^*(t)$, and we can also show that $v_2^*\leq0$ and bounded. In fact,
\begin{equation*}
v_2^*(t,x,y)=-\eta(t,x,y) w_y(t,x,y) =\frac{ak}{\gamma-1}\frac{g_y(t,y)}{g(t,y)}.
\end{equation*}
Due to $a, k>0$, $\gamma >1$ and expression of $g$ in the following Theorem \ref{th1}, we have
\begin{equation}\label{pz19}
\frac{g_y(t,y)}{g(t,y)}
=-\frac{\int_t^Te^{A(t,s)-B(t,s)}B(t,s)ds}{\int_t^Te^{A(t,s)-B(t,s)}ds}\leq 0
\end{equation}
and $v_2^*\leq0$. Moreover,  by \eqref{pz19},
\begin{equation}\label{eqofg1}
|\frac{g_y(t,y)}{g(t,y)}| \leq |B|,\ t \in [0,T], y \in [0,\infty),
\end{equation}
where $|B|:=\max_{s,t \in [0,T], t \leq s}|B(t,s)| $ is bounded which will be proved in \eqref{eqofab}. So $v_2^* \in \mathcal{A}_t$.


\end{Remark}

In this section, the following result provides an explicit solution to the robust consumption-portfolio problem of an investor with Heston model.

\begin{theorem}\label{th1}
Assume that (i)(ii) in Proposition \ref{pz5},  Assumptions \ref{gamma}, \ref{assume2} and (H1)-(H3) hold. Then the investor's value function $w$ in the Heston stochastic volatility model has the representation \eqref{eq13}, where
\begin{equation*}\label{eqofgg}
g(t,y)=\delta^\psi \int_t^Te^{A(t,s)-B(t,s)y}ds
\end{equation*}
and the functions $A$ and $B$ are explicitly given by
 \begin{equation}\label{eqab}
 \begin{split}
 A(t,s)= &\frac{4b\nu}{{\kappa}^2-d^2} \Big[\ln(({\kappa}+d)e^{ds}-({\kappa}-d)e^{dt})-\ln(2d)-\frac{1}{2}(({\kappa}+d)s-({\kappa}-d)t)\Big]\\
&+\frac{(1-\gamma)r-\delta\theta}{k}(s-t),\\
B(t,s)= &2b\frac{e^{d(s-t)}-1}{e^{d(s-t)}({\kappa}+d)-{\kappa}+d},
 \end{split}
\end{equation}
where $d:=\sqrt{{\kappa}^2+ 2b\bar{\beta}^2}$, $b$ and $\kappa$ are defined in (H3).
The optimal portfolio strategy and consumption-wealth ratio read
\begin{equation}\label{solution2}
\pi^*_s=\frac{\bar{\lambda}}{(\gamma+a)}+\frac{(1-\gamma-a)k\bar{\beta}\rho}{(\gamma+a)(1-\gamma)}\frac{g_y(s,Y_s)}{g(s,Y_s)}\text{ and  } (\frac{c_s}{X_s})^*=\frac{\delta^\psi}{g(s,Y_s)}, ~s \in [t,T].
\end{equation}
\end{theorem}

\begin{proof}
The proof is divided into three steps. We refer to Appendix C in Kraft, Seifried and Steffensen \cite{Kraft} for the proof of Steps 1 and 2, so we only give the arguments related to our robust model and omit the similar arguments in the first two steps.

Step 1:  Explicit solution of \eqref{eq18a}.
In the Heston model, the PDE \eqref{eq19c} for $h$ becomes
 \begin{equation}\label{eq22}
 h_t+\frac{1}{k}\big[(1-\gamma)r -\delta \theta + \frac{1-\gamma}{2(\gamma+a)}\bar{\lambda}^2y \big]h+\big[\nu+ (\frac{1-\gamma -a}{\gamma + a}\rho\bar{\lambda}\bar{\beta}-m)y\big]h_y+\frac{1}{2}\bar{\beta}^2yh_{yy}=0.
 \end{equation}
 We conjecture that the solution $h$ has the following form:
 \begin{equation}\label{eq23}
 h(t,y;s)=e^{A(t,s)-B(t,s)y}
 \end{equation}
 with $h(s,y;s)=1$.
 Substituting \eqref{eq23} into \eqref{eq22}, we get
 \begin{equation}\label{eq24}
 \begin{split}
 &dB(t,s)=\Big[(m- \frac{1-\gamma-a}{\gamma+a}\rho\bar{\lambda}\bar{\beta})B+ \frac{1}{2}\bar{\beta}^2B^2+ \frac{1-\gamma}{2k(\gamma+a)}\bar{\lambda}^2\Big]dt,\quad B(s,s)=0,\\
 & dA(t,s) =\Big[\nu B -\frac{1}{k}[(1-\gamma)r-\delta\theta]\Big]dt, \quad A(s,s)=0.
 \end{split}
 \end{equation}
 It is easy to get that \eqref{eqab}  solves above ODE \eqref{eq24}.
  The candidate optimal strategy \eqref{solution2} follows from the results \eqref{eq16}, and we shall check that $(c^*, \pi^*)$ is admissible.
  For $\pi^*$, by its explicit expression, together with $\rho <0$ and $g_y \leq 0$, we can get
\begin{equation*}
\frac{\bar{\lambda}}{\gamma+a} \leq \pi^*(s,y) \leq K_\pi,\ \ \ s\in [t,T],\ y \in [0,\infty),
\end{equation*}
where $K_\pi$ is defined below (H3).
For $c^*$, we claim
\begin{equation*}
(\frac{c}{X})^*(s,y) \leq \frac{1}{T-s} + by +\frac{1}{2}\nu b (T-s),\ \ \ s \in [t,T],\ y \in [0,\infty),
\end{equation*}
whose proof is similar to Lemma C.5 in \cite{Kraft}.
%
%

Step 2: Verification of the condition that $\int_t w_z(s, Z_s)^\top\Lambda(s, Z_s;\pi_s)d{\bf{B}}_s$ is a true martingale for $(c,\pi;v)\in\mathcal{A}_t$.
From \eqref{eqab}, $A$ and $B$ satisfy
\begin{equation}\label{eqofab}
\begin{aligned}
0 \leq & B(t,s) \leq \frac{b}{\kappa}, ~\text{ for all } t \leq s,\\
-\frac{b}{\kappa}\nu T -|\frac{(1-\gamma)r-\delta\theta}{k}| T \leq &A(t,s) \leq |\frac{(1-\gamma)r-\delta\theta}{k}| T,   ~\text{ for all } t \leq s,
\end{aligned}
\end{equation}
and similar to Proposition C.6 in \cite{Kraft}, we have
\begin{equation}\label{pz11}
\mathbb{E}[X_s^{q(1-\gamma)}] \leq C(T-s)^{q(1-\gamma)},
\end{equation}
for a positive and sufficiently small $q-2$. Here and in the rest of the proof, $C$ is a generic constant depending only on given parameters and its values may change from line to line. Then we attain the goal of Step 2 by
a similar proof as Proposition C.3 in \cite{Kraft}, based on the state equation of $Z$ in \eqref{eq9} and estimates of  \eqref{eqofg1}, \eqref{eqofab} and \eqref{pz11}.

Step 3: Verification of the condition that $\mathbb{E}[\int_t^T|w(s,Z_s)-I_s|ds] < \infty$ for $(c,\pi;v)\in\mathcal{A}_t$. 
We first show
\begin{equation}\label{pz17}
\mathbb{E}[\int_t^T|w(s,X_s,Y_s)|]<\infty,
\end{equation}
 for all $(c,\pi;v) \in \mathcal{A}_t$. In fact,
\begin{equation}\label{pz15}
\mathbb{E}[\int_t^T |w(s,X_s,Y_s)|ds]
=\mathbb{E}[\int_t^T |\frac{1}{1-\gamma}X_s^{1-\gamma}g^k(s,Y_s)|ds].
\end{equation}
By \eqref{robustwealthofx} and $v_1=-\frac{a}{x}$ in $\mathcal{A}_t$, we have
\begin{equation*}
X_t=x\exp{\{ \int_0^t [(r+\pi_s\bar{\lambda}Y_s - C_s) - \pi_s^2 aY_s+ \pi_s \rho\bar{\beta}v_2(s)Y_s -\frac{1}{2}\pi_s^2Y_s]ds+\int_0^t \pi_s\sqrt{Y_s}dB_s\}}\geq 0.
\end{equation*}
On the other hand, since $B(t,s)\geq0$ and $Y_s\geq0$, $0\leq t \leq s$, we have
$0 \leq  g(t,y) \leq e^{|A|}(T-t),\ t \in [0,T],\ y \in [0,\infty)$, where $|A|:=\max_{s,t \in [0,T],\ t \leq s}|A(t,s)|$. Hence  we can obtain from \eqref{pz11} that, for a positive and sufficiently small $q-2$,
\begin{equation*}
\begin{aligned}
& \mathbb{E}[X_s^{1-\gamma} g^k(s,Y_s)]\\
\leq & C \mathbb{E}[X_s^{1-\gamma}(T-s)^k]\\
\leq & C \{\mathbb{E}[X_s^{q(1-\gamma)}(T-s)^{qk}]\}^\frac{1}{q}\\
\leq & C (T-s)^{1-\gamma+k}.
\end{aligned}
\end{equation*}
Moreover, noticing from (H1) that $1-\gamma+k >-1$, we have
\begin{equation}\label{pz16}
\int_t^T\mathbb{E}[X_s^{1-\gamma} g^k(s,Y_s)]ds \leq C \int_t^T(T-s)^{1-\gamma+k}ds < \infty.
\end{equation}
Consequently, \eqref{pz17} follows from \eqref{pz15} and \eqref{pz16}.

Next we prove
\begin{equation}\label{pz20}
\mathbb{E}[\int_t^T I_sds]< \infty,
\end{equation}
for all $(c,\pi; v) \in \mathcal{A}_t$.
Note that the cost functional $I$ defined in \eqref{pz3} is the solution of the following BSDE
\begin{equation*}
dI_t=-\bar{g}(t,Z_t,u_t,v_t,I_t)dt+ L_t dB_t,\ I_T=0,
\end{equation*}
where
\begin{equation*}
\bar{g}(t,z,u,v,I)=f(c,I)+\frac{1}{2\eta(t,x,y)}v^\top_t\Sigma(t,x,y,\pi)v_t,
\end{equation*}
and
\begin{equation*}
z=(x,y)^\top, u=(c,\pi)^\top, v=(v_1,v_2)^\top.
\end{equation*}
By the fact that $\gamma >1$ and Assumption 3.1, $\bar{g}$ is monotonic with respect to  $I$,
\begin{equation*}\label{pz10}
(I_1-I_2)(\bar{g}(s,z,u,v,I_1)-\bar{g}(s,z,u,v,I_2)) \leq |\delta\theta||I_1-I_2|^2,
\end{equation*}
and
\begin{equation*}
\frac{I}{|I|}1_{|I|\neq 0}\bar{g}(s,z,u,v,I)\leq |\delta \theta|I + |\frac{1}{2\eta(s,x,y)}v^\top_s\Sigma(s,x,y,\pi)v_s|.
\end{equation*}
According to Proposition 3.2 in \cite{Briand}, we have
\begin{equation}\label{pz13}
\mathbb{E}[\int_t^T|I_s|^p]ds \leq C \mathbb{E}[\int_t^T ( \frac{1}{2\eta(s,X_s,Y_s)}v^\top_s\Sigma(s,X_s,Y_s,\pi_s)v_s)^p ds],\ \ \  p>1.
\end{equation}
Substituting the definition of $\eta$, $\Sigma$, $w$ into $\frac{1}{2\eta(s,x,y)}v^\top_s\Sigma(s,x,y,\pi)v_s$,
we have
\begin{equation*}
\begin{aligned}
&\frac{1}{2\eta(s,x,y)}v^\top_s\Sigma(s,x,y,\pi)v_s\\
=&\frac{1}{2a_1}g^k(s,y)(a_1^2\pi^2-2a_1\pi\rho
\bar{\beta}v_2+\bar{\beta}^2v_2)x^{1-\gamma}y\\
\leq& C x^{1-\gamma}y.
\end{aligned}
\end{equation*}
The last inequality based on the facts that the admissible control  $\pi$ and $v_2$ are bounded, $g$ is bounded, $\rho \leq 0$ and $\bar{\beta}>0$. Then
\begin{equation}\label{pz9}
\begin{aligned}
&\mathbb{E}[\int_t^T (\frac{1}{2\eta(s,X_s,Y_s)}v^\top_s\Sigma(s,X_s,Y_s,\pi_s)v_s)^pds]\\
\leq & C \mathbb{E}[\int_t^T (X_s^{1-\gamma}Y_s)^pds]\\
\leq & C \{\mathbb{E}[\int_t^T X_s^{2p(1-\gamma)}ds]\}^{1/2}\{\mathbb{E}[\int_t^TY_s^{2p}ds]\}^{1/2},\ \ \ p>1.
\end{aligned}
\end{equation}

To get the integration of $Y$, first note
\begin{equation*}
dY_s=(\nu-k_1(s) Y_s)ds+ \bar{\beta}\sqrt{Y_s}(\rho dB_s+\sqrt{1-\rho^2}d\hat{B}_s),
\end{equation*}
where
\begin{equation*}
k_1(s)=m+a_1\rho\bar{\beta}\pi_s-\bar{\beta}^2 v_2(s).
\end{equation*}
Since
\begin{equation*}
\pi {\rm\text \ is \ bounded\  by\ } K_\pi~{\rm\text and}~ v_2\leq0,
\end{equation*}
we have
\begin{equation*}
k_1(s)>m-a_1\bar{\beta}K_\pi.
\end{equation*}
Define
\begin{equation*}
k_2:=m-a_1\bar{\beta}K_\pi.
\end{equation*}
By (H2) we know $k_2 >0$, and thus
\begin{equation*}
\nu-k_1(s)y < \nu-k_2 y ~~\text{for all}  ~s \in [t,T] ,\  y \in [0,\infty).
\end{equation*}
Hence by the comparison theorem, it follows that
\begin{equation}\label{comparisonofy}
Y_s \leq \tilde{Y}_s ~~\text{for all } s \in [t,T],
\end{equation}
where $\tilde{Y}$ is the solution of the following
\begin{equation*}\label{eqoftildey}
d\tilde{Y}_s=(\nu-k_2 \tilde{Y}_s)ds+ \bar{\beta}\sqrt{\tilde{Y}_s}(\rho d\tilde{B}_s+\sqrt{1-\rho^2}d\hat{B}_s),~~\tilde{Y}_t =y.
\end{equation*}
By Pitman-Yor Lemma for Laplace transform of the integrated
square-root process in \cite{pit-yor}, we have $\sup_{t \in [0,T]}\mathbb{E}[\tilde{Y}_t^{p_1}] < \infty$ for any $p_1>1$, which together with \eqref{comparisonofy} leads to $0 \leq  \sup_{t \in [0,T]}\mathbb{E}[Y_t^{p_1}]  < \infty$ for any $p_1>1$.
Hence for a positive and sufficiently small $p-1$,
\begin{equation}\label{expectationofy3}
\mathbb{E}[\int_t^TY_s^{2p}ds] < \infty.
\end{equation}
As for the integration of $X$, by \eqref{pz11}
we have, for a positive and sufficiently small $q-2$,
\begin{equation*}\label{add1}
\int_t^T \mathbb{E}[X_s^{q(1-\gamma)}]ds \leq C\int_t^T (T-s)^{q(1-\gamma)} ds.
\end{equation*}
By (H1), $\gamma < \frac{1}{q}+1$, i.e. $q(1-\gamma)+1 >0$. Then we know that   $\int_t^T (T-s)^{q(1-\gamma)} ds$ is integrable, and
\begin{equation*}
\int_t^T (T-s)^{q(1-\gamma)} ds=\frac{1}{q(1-\gamma)+1}T^{q(1-\gamma)+1}.
\end{equation*}
Thus, for a positive and sufficiently small $p-1$,
\begin{equation}\label{pz22}
\mathbb{E}[\int_t^T X_s^{2p(1-\gamma)}ds]\}< \infty.
\end{equation}
By  \eqref{pz13}, \eqref{pz9}, \eqref{expectationofy3} and \eqref{pz22}, for a positive and sufficiently small $p-1$, we have
 \begin{equation*}
 \mathbb{E}[\int_t^T |I_s|^pds] < \infty,
 \end{equation*}
which implies (\ref{pz20}).


Therefore, Step 3 is derived from (\ref{pz17}) and (\ref{pz20}).
\end{proof}

Next, we compare the results of robust consumption-portfolio with those of non-robustness. For this, set
\begin{equation}\label{eq24a}
\begin{array}{ccccc}
 \gamma=1.4,                        & \delta=0.08, & \rho=-0.5,        &r=0.05,  \\
\sqrt{\bar{y}}=0.15,  & \bar{\lambda}\sqrt{\bar{y}}=0.07,  & m=5,    & \bar{\beta}=0.25. &
\end{array}
\end{equation}



\begin{center}
\includegraphics[width = 12cm, height = 7cm]{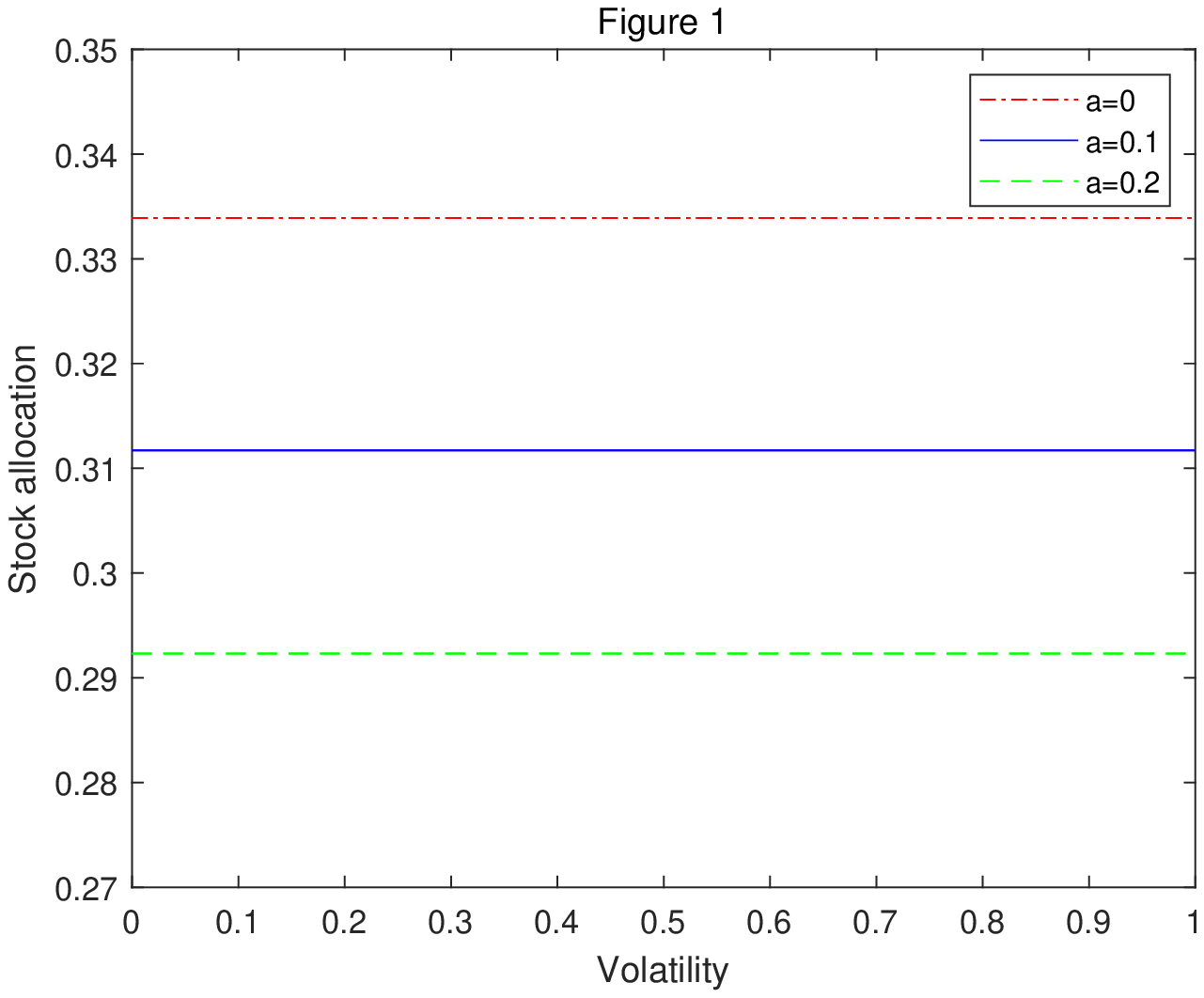}
\end{center}
\begin{center}
\includegraphics[width = 12cm, height = 7cm]{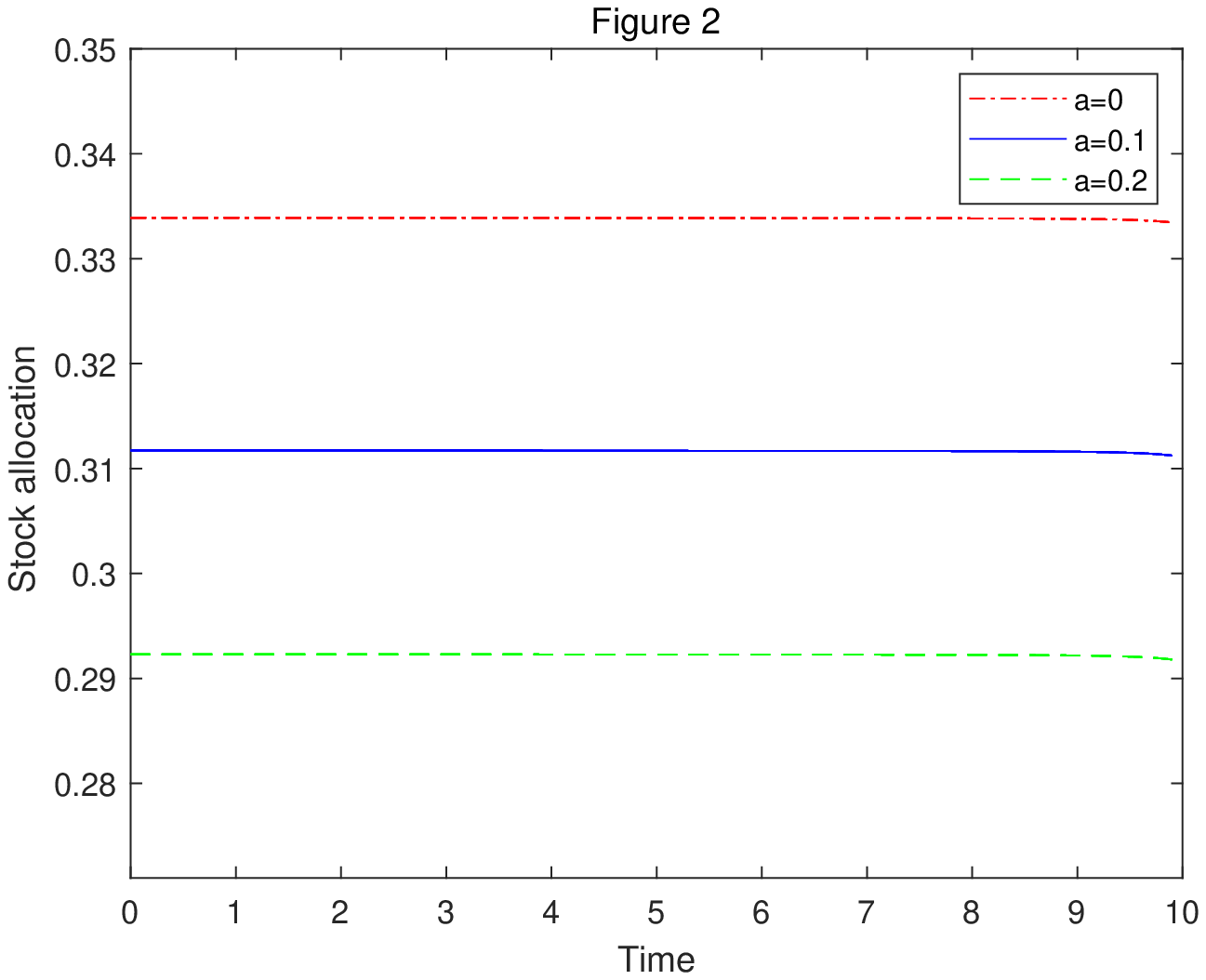}
\end{center}
\begin{center}
\includegraphics[width = 12cm, height = 7cm]{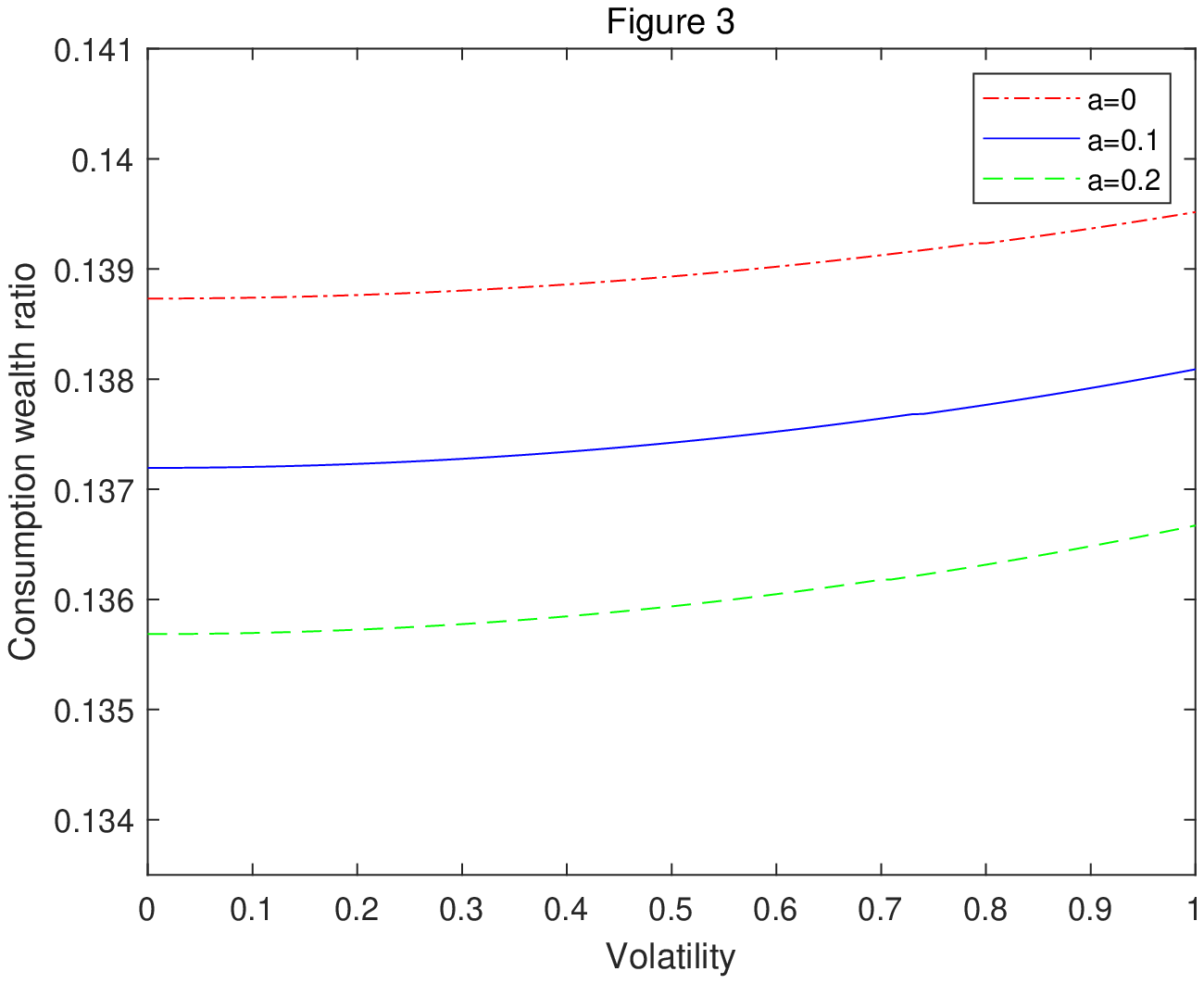}
\end{center}

Figures $1$, $2$ and $3$ demonstrate the results of Theorem \ref{th1} for a time horizon of $T = 10$ years and the parameters in (\ref{eq24a}). Also we provide the portfolio and consumption strategies with robust ($a=0.1$, $a=0.2$) and non-robust ($a=0$) investors.

Figure $1$ shows the optimal stock allocation as a function of volatility at time $t = 0$, while Figure $2$ shows the optimal stock allocation as a function of time when the volatility is given  $0.2$.
In Figure $1$, we find in both cases of non-robustness ($a = 0$) and robustness ($a = 0.1$, $a=0.2$), the optimal stock allocation is insensitive to the volatility at initial time and only decreases very slightly in volatility. However, more robust investor is more cautious. He/She puts lower optimal stock allocation than less robust investors.

In Figure $2$, we consider the investment in the horizon of $10$ years if the volatility is $0.2$. From Figures $1$ and $2$, we find not only at the beginning of stock allocation, but also in the following operation, more robust investors put lower optimal stock allocation than the non-robust investors.

Figure $3$ demonstrates the difference of the optimal consumption-wealth ratio between the robust and non-robust investors as a function of volatility at time $t = 0$. 
We find that, at the beginning of consumption, less robust investors are willing to consume.

According to above analysis, what will robust investors do while they put lower proportion of the wealth into the consumption and stock? In fact, robust investors are willing to put their money into the riskless financial markets, such as banks and bond markets. We find their behaviors showed in our figures are identical with our common sense. Robust investors are those who do not totally trust their reference model and prefer to choose the worst scenario as their utilities. Consequently, robust investors are more pessimistic than the others and their behaviors are more conservative than the non-robust investors.

\begin{Remark}
We provide 3 corresponding figures to show the trend of the risk-averse investor by different $\gamma$ with Heston model in comparison with the robust investor in Figures 1-3 by setting
\begin{equation*}
\begin{array}{ccccc}
 a=0,                        & \delta=0.08, & \rho=-0.5,        &r=0.05,  \\
\sqrt{\bar{y}}=0.15,  & \bar{\lambda}\sqrt{\bar{y}}=0.07,  & m=5,    & \bar{\beta}=0.25. &
\end{array}
\end{equation*}
\begin{center}
\includegraphics[width = 8cm, height = 4.67cm]{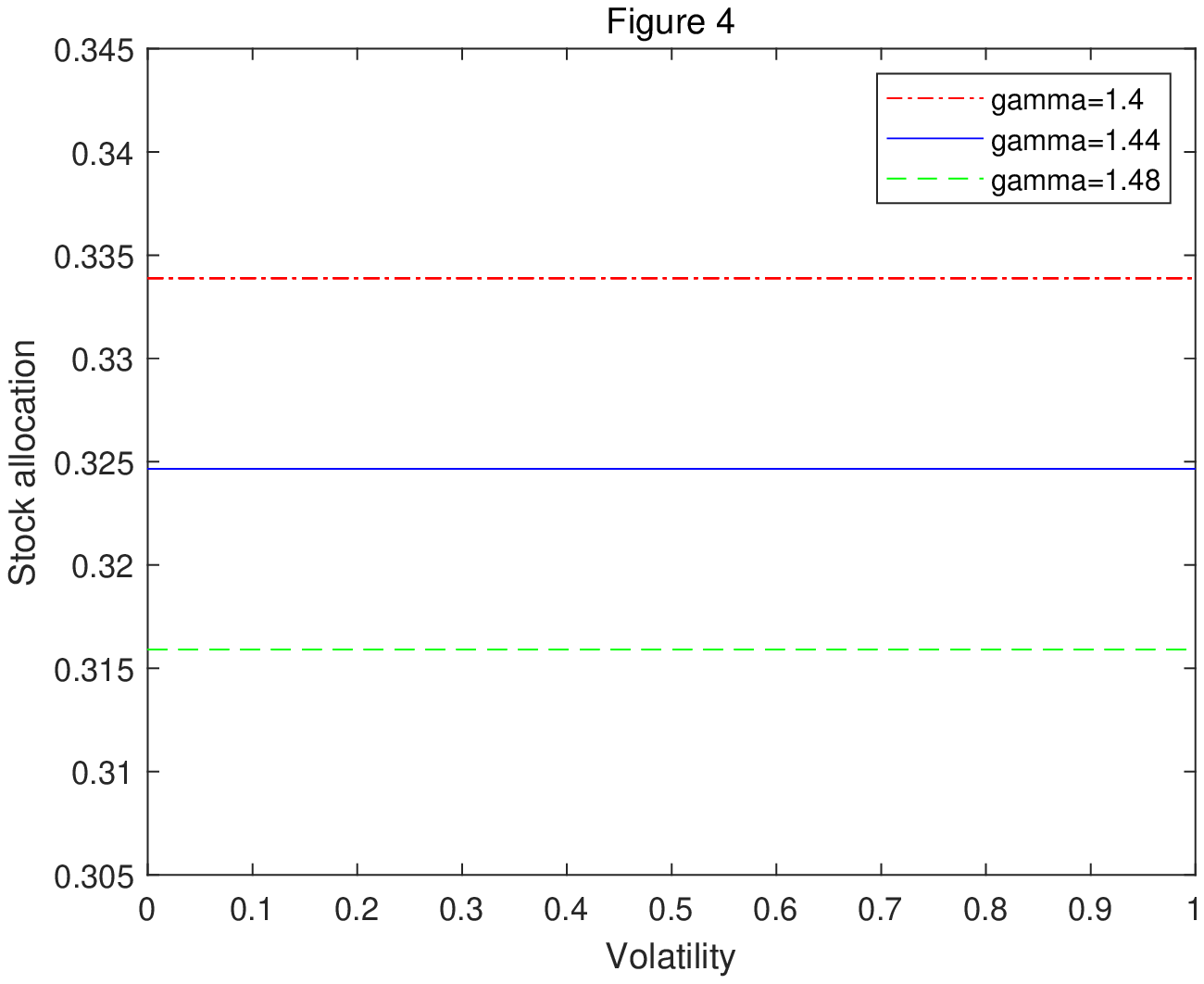}\includegraphics[width = 8cm, height = 4.67cm]{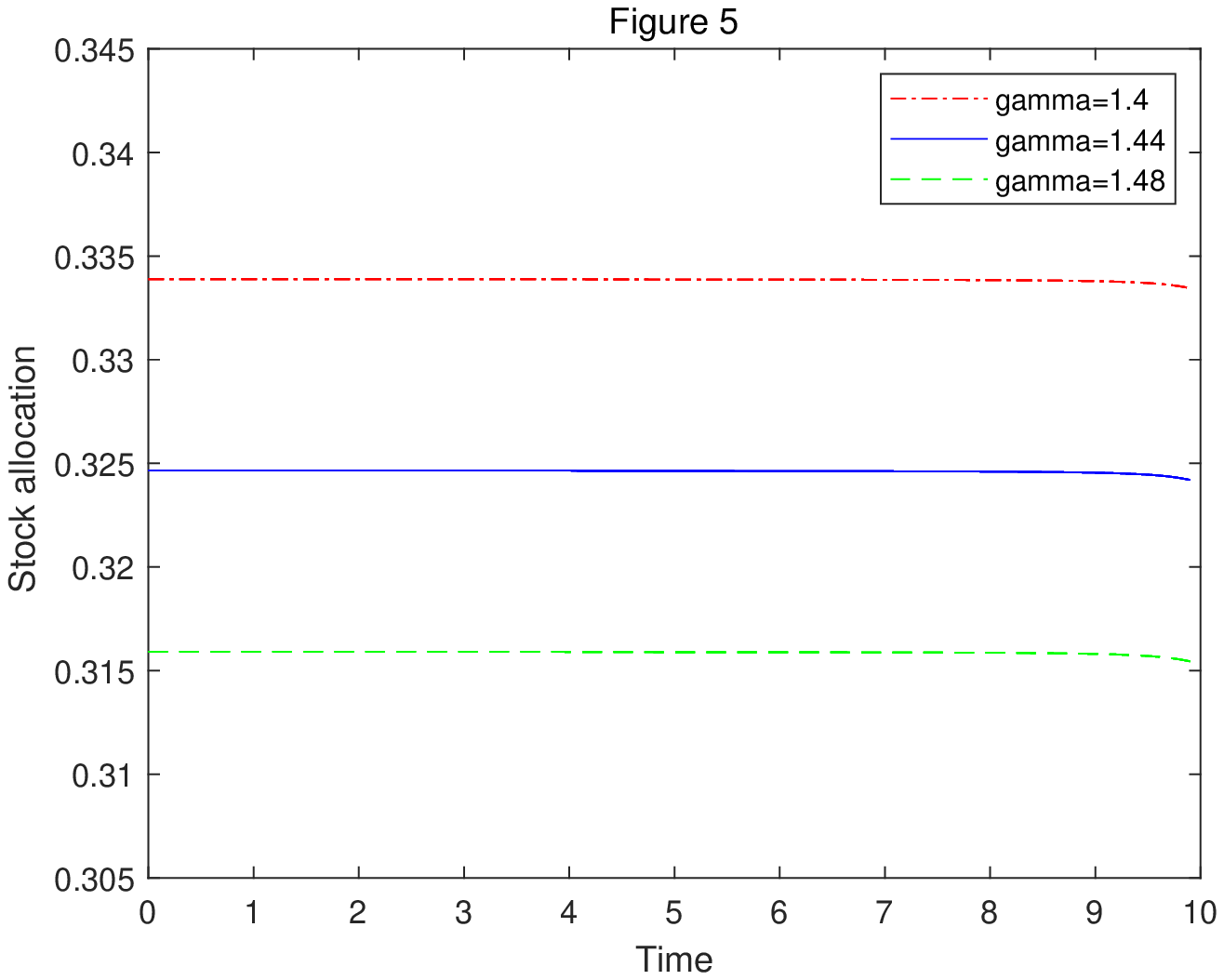}
\end{center}
\begin{center}
\includegraphics[width = 8cm, height = 4.67cm]{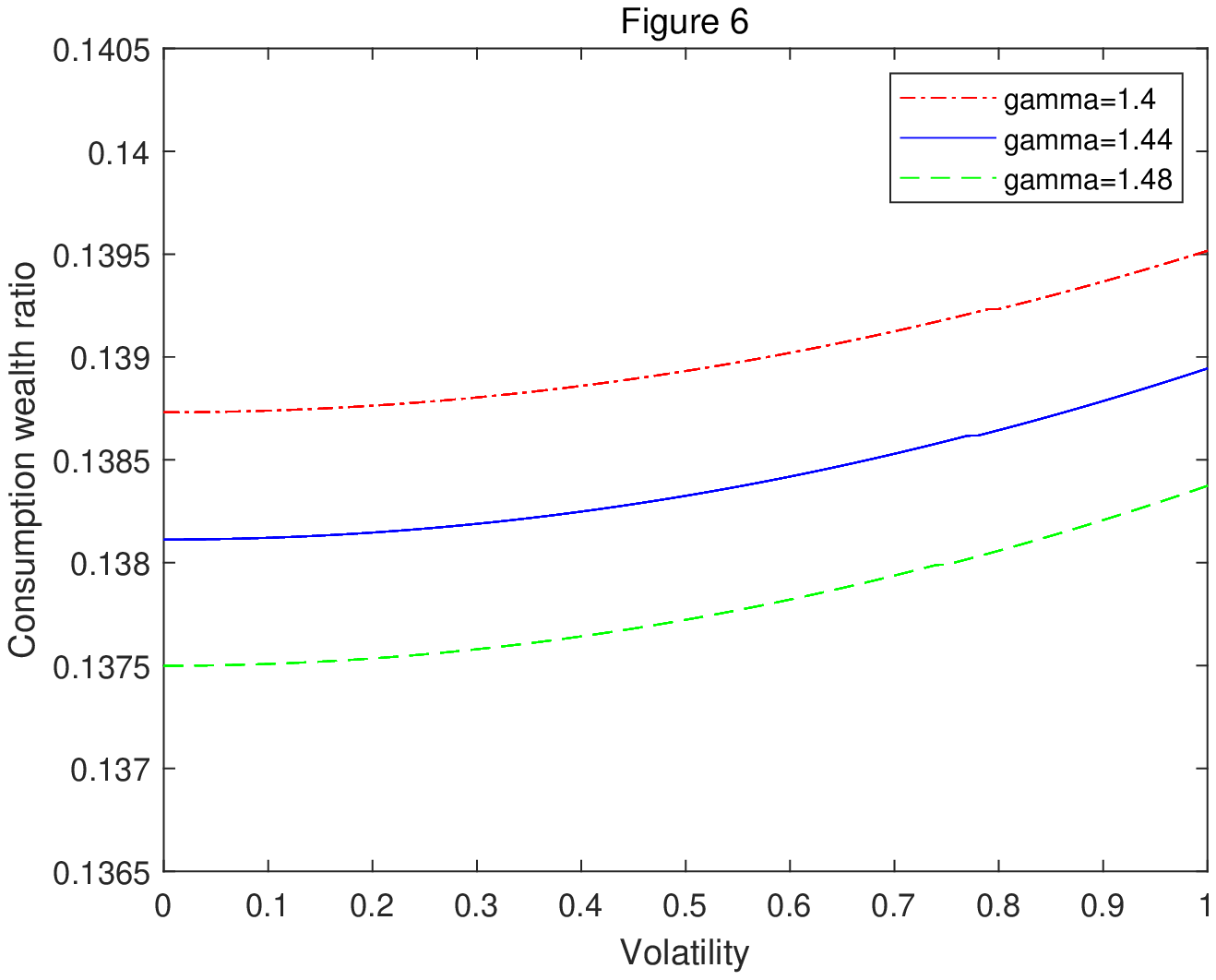}
\end{center}
The simulation results show that the outcomes of the robust investor have very similar trends to 
those of the risk-averse investor, although there maybe exist some little differences due to parameter settings. In fact, it is pointed out in \cite{Maenhout06} and \cite{Liu} that the portfolio rule of a robust investor with risk aversion $\gamma$ and preference for robustness $a$ is identical to the one of a non-robust investor with risk aversion $\gamma+a$ for the mean-reverting model. We believe from the simulation that a similar result can be obtained  for our model as the mean-reverting model.
\end{Remark}


\section*{Acknowledgements} The authors would like to thank the anonymous referees for their valuable comments which improve this paper much.

\end{document}